\documentclass[12pt,amscd]{amsart}
\usepackage[pagebackref=true,colorlinks=true,linkcolor=blue,citecolor=blue]{hyperref}

\newtheorem{thm}{Theorem}[section]
\usepackage[all]{xy}
\usepackage{graphicx,tikz-cd}
\usepackage{amsmath,amsxtra,amssymb,latexsym, amscd,amsthm}
\usepackage{indentfirst}
\usepackage[mathscr]{eucal}
\usepackage{hyperref}
\hypersetup{colorlinks=true,linkcolor=red,filecolor=magenta,urlcolor=green}

\newtheorem{lem}[thm]{Lemma}

\theoremstyle{definition}
\newtheorem{defn}[thm]{Definition}
\newtheorem{exam}[thm]{Example} 
\newtheorem{rem}[thm]{Remark}

\DeclareMathOperator{\hdim}{hdim}
\DeclareMathOperator{\SP}{\mathcal{SP}}

\DeclareMathOperator{\supp}{supp}

\DeclareMathOperator{\lk}{lk}
\DeclareMathOperator{\reg}{reg}
\DeclareMathOperator{\pd}{pd}

\DeclareMathOperator{\Min}{Min}

\def\b {\mathbf b}

\def\v {\mathbf v}
\def\a {\mathbf a}

\def\k {\mathrm k}

\def\h {\widetilde{H}}

\def\h {\widetilde{H}}

\def\NN {\mathbb{N}}

\begin{document}

\title[Regularity of symbolic powers of cover ideals]{Critical subgraphs and the regularity of symbolic powers of cover ideals of graphs}

\author[N.T. Hang]{Nguyen Thu Hang }
\address{Thai Nguyen University of Sciences, Phan Dinh phung Ward, Thai Nguyen, Vietnam}
\email{hangnt@tnus.edu.vn}

\author{Thanh Vu}
\address{Institute of Mathematics, VAST, 18 Hoang Quoc Viet, Hanoi, Vietnam}
\email{vuqthanh@gmail.com}

\dedicatory{Dedicated to Professor Nguyen Tu Cuong on the occasion of his 75th birthday}

%\commby{}
%-----------------------------------------------------------
\begin{abstract} 
Let $G$ be a simple graph. We demonstrate a method for using $t$-admissible subgraphs of $G$ to determine the regularity of the $t$-th symbolic power of the cover ideal of $G$. As an application, we  compute the regularity of powers of cover ideals of bipartite unicyclic graphs.
\end{abstract}

% -----------------------------------------------------------
\maketitle
% -----------------------------------------------------------
\section{Introduction}
Let $G$ be a simple graph. In \cite{DHV}, Dung, Hang, and Vu introduced the notion of \emph{admissible subgraphs} and used it to compute the depth of symbolic powers of cover ideals of graphs. This approach is made possible by Hochster's depth formula \cite{H} together with the observation that the associated radicals of symbolic powers of cover ideals of graphs correspond to the cover ideals of certain subgraphs of $G$, which they called admissible subgraphs. In this work, we demonstrate how to use admissible subgraphs to study the regularity of symbolic powers of cover ideals of graphs. We now introduce some notation.

Let $S = \k[x_1, \ldots, x_n]$ be a standard graded polynomial ring over a field $\k$, and let $I$ be a homogeneous ideal of $S$. Denote by $\Min(I)$ the set of minimal associated primes of $I$. Then the $t$-th symbolic power of $I$, denoted by $I^{(t)}$, is defined by
\[
I^{(t)} = \bigcap_{\mathfrak{p} \in \Min(I)} I^t S_{\mathfrak{p}} \cap S.
\]
By a result of Cutkosky, Herzog, Trung \cite{CHT} and independently of Kodiyalam \cite{K}, the regularity function of the ordinary powers of $I$ is eventually linear. On the other hand, the regularity of the symbolic powers of $I$ can exhibit much more complicated behavior \cite{Cut}. When $I$ is a squarefree monomial ideal, Herzog, Hibi, and Trung \cite{HHT} showed that the regularity function of $I^{(t)}$ is eventually quasi-linear. 

We now introduce the edge ideal and the cover ideal of a graph, which are the main objects of study in this work. Let $G$ be a graph with vertex set $V(G) = \{1, \ldots, n\}$ and edge set $E(G)$. The \textit{edge ideal} and the \textit{cover ideal} of $G$, denoted by $I(G)$ and $J(G)$, respectively, 
are defined by
\[
I(G) = (x_i x_j \mid \{i,j\} \in E(G)), 
\qquad 
J(G) = \bigcap_{\{i,j\} \in E(G)} (x_i, x_j).
\]

The study of the regularity of ordinary and symbolic powers of $I(G)$ and $J(G)$ is an active area of research. Despite this, the precise behavior of these functions for general graphs remains unknown. For edge ideals, most recent work focuses on Minh's conjecture, which states that $\reg(I(G)^t) = \reg(I(G)^{(t)})$. In particular, it is not yet known whether $\reg(I(G)^{(t)})$ is eventually linear. Explicit calculations have shown that Minh's conjecture holds for certain classes of graphs; for a survey on this topic, see \cite{MV1}. On the other hand, Dung, Hien, Nguyen, and Trung \cite{DHNT} recently demonstrated that the regularity of symbolic powers of cover ideals of graphs 
might not be eventually linear. The study of the regularity of powers of cover ideals has largely centered on the Herzog--Hibi--Ohsugi conjecture \cite{HHO}, which asserts that all powers of the cover ideal of a chordal graph are componentwise linear. This conjecture has led to explicit formulas for the regularity of ordinary and symbolic powers of cover ideals for several classes of graphs. We refer to the survey article \cite{HT} and the references therein for the current status of the conjecture.

Recently, Hien and Trung \cite{HiT} established a sharp bound for the regularity of symbolic powers of squarefree monomial ideals. Building on this result, Fakhari \cite{F3} determined the regularity of symbolic powers of cover ideals of Cameron--Walker graphs and of claw-free graphs with no cycles of lengths other than $3$ and $5$. However, the regularity of symbolic powers of cover ideals of cycles remains unknown.

Minh, Nam, Phong, Thuy, and Vu \cite{MNPTV} showed how to compute the regularity of a monomial ideal via its associated radicals and proved that Minh's conjecture holds for the second and third powers. These methods have since been used to verify the conjecture for other classes of graphs and to study the regularity of powers of edge ideals \cite{MV2}. In this work, we further demonstrate the usefulness of this approach in studying the regularity of symbolic powers of cover ideals of unicyclic graphs. We note that this problem is closely related to the study of the homological dimension of the simplicial complex associated with the cover ideal. The following result on the regularity of the cover ideal of a unicyclic graph is new. For a graph $G$, we denote by $\tau_{\max}(G)$ the maximum cardinality of a minimal vertex cover of $G$.

\begin{thm}\label{thm_reg_unicycli}
Let $G$ be a unicyclic graph whose unique cycle is $C_n$. Then
\[
\reg\bigl(J(G)\bigr)
=
\begin{cases}
\tau_{\max}(G) + 1 & \text{if } n \equiv 1 \pmod 3 \text{ and } G \text{ is reducible to } C_n,\\
\tau_{\max}(G) & \text{otherwise}.
\end{cases}
\]
\end{thm}

For powers, we prove the following result.

\begin{thm}\label{main}
Let $G$ be a unicyclic graph whose unique cycle is $C_n$ with $n$ even. Then, for all $t \ge 2$,
\[
\reg\bigl(J(G)^t\bigr)
=
\begin{cases}
t\,\tau_{\max}(G) + 1 & \text{if } n = 4 \text{ and } G \text{ is reducible to } C_4,\\
t\,\tau_{\max}(G) & \text{otherwise}.
\end{cases}
\]
\end{thm}

The notion of a graph being reducible to $C_n$ is adapted from the work of Pham and Vu \cite{PV}; we refer to Section~\ref{sec_3} for further details. For symbolic powers, we actually prove a stronger result, which includes the corresponding statement for symbolic powers of cover ideals of other non-bipartite unicyclic graphs, including all odd cycles. It is closely related to the concept of asymptotic cover degree, for which we also refer to Section~\ref{sec_3} for further details.

Section~\ref{sec_2} is devoted to developing the machinery needed to use admissible subgraphs to compute the regularity of symbolic powers of cover ideals of graphs. In Section~\ref{sec_3}, we apply this machinery to prove Theorem~\ref{thm_reg_unicycli} and Theorem~\ref{main}.

\section{Critical subhypergraphs}\label{sec_2}
In this section, we recall basic notation and results used throughout the paper. We then introduce the notion of \( t \)-critical subhypergraphs and explain how to use them to compute the regularity of symbolic powers of cover ideals of hypergraphs. Throughout the paper, let \( \k \) be a field, \( S = \k[x_1, \ldots, x_n] \) a polynomial ring, and \( \mathfrak{m} = (x_1, \ldots, x_n) \) the \emph{maximal homogeneous ideal} of \( S \).

\subsection{Simplicial complexes and Stanley-Reisner correspondence}
Let $\Delta$ be a simplicial complex on $[n] = \{1, \ldots, n\}$; that is, a collection of subsets of $[n]$ that is closed under taking subsets. It is clear that $\Delta$ is uniquely determined by its maximal elements under inclusion, called \emph{facets}. The set of facets of $\Delta$ is denoted by $\mathcal{F}(\Delta)$. For a facet $F$ of a simplicial complex $\Delta$, we denote by $\Delta \setminus F$ the simplicial complex whose set of facets is
\[
\mathcal{F}(\Delta \setminus F) =  \mathcal{F}(\Delta) \setminus \{F\}.
\]

For a face $F \in \Delta$, the link of $F$ in $\Delta$ is defined by 
$$ \lk_{\Delta} F = \{G \in \Delta \mid F \cup G \in \Delta, F \cap G = \emptyset \}.$$

Let $\Delta_1$ and $\Delta_2$ be simplicial complexes. The \emph{join} of $\Delta_1$ and $\Delta_2$, denoted by $\Delta_1 * \Delta_2$, is the simplicial complex on the vertex set $V(\Delta_1) \sqcup V(\Delta_2)$ whose facets are precisely the sets $F \cup G$, where $F$ is a facet of $\Delta_1$ and $G$ is a facet of $\Delta_2$.

\begin{defn}
Let $\Delta$ be a simplicial complex.
\begin{enumerate}
    \item The $q$-th reduced homology group of $\Delta$ with coefficients in $\k$, denoted by $\h_q(\Delta;\k)$, is the $q$-th homology group of the augmented oriented chain complex of $\Delta$ over $\k$.
    
    \item A simplicial complex $\Delta$ is called \emph{acyclic} if $\h_i(\Delta;\k) = 0$ for all $i$.
    
    \item If $\Delta$ is not acyclic, then the \emph{homological dimension} of $\Delta$ over $\k$, denoted by $\hdim_{\k}(\Delta)$, is defined by
    \[
    \hdim_{\k}(\Delta)
    =
    \max \{\, i \mid \h_i(\Delta;\k) \neq 0 \,\}.
    \]
\end{enumerate}
\end{defn}
\begin{rem} Let $\Delta$ be a simplicial complex. Then \begin{enumerate}
    \item $\h_{-1}(\Delta;\k) \neq 0$ if and only if $\Delta$ is the empty complex (i.e., $\Delta=\{\emptyset\}$).
    \item If $\Delta$ is a cone over some $t \in [n]$ or $\Delta$ is the void complex  (i.e., $\Delta=\emptyset$), then it is acyclic.
\end{enumerate}
\end{rem}

For each subset $F$ of $[n]$, let $x_F = \prod_{i\in F} x_i$ be a squarefree monomial of $S$. We now recall the Stanley-Reisner correspondence. 
\begin{defn}
    
\begin{enumerate}
    \item For a squarefree monomial ideal $I \subseteq S$, the Stanley-Reinser complex of $I$ is defined by $\Delta(I) = \{ F \subseteq [n] \mid x_F \notin I\}$. 
    \item For a simplicial complex $\Delta$ on the vertex set $[n]$, the Stanley-Reisner ideal of $\Delta$ is defined by $I_\Delta = (x_F \mid F \notin \Delta)$. 
\end{enumerate}
\end{defn}

\subsection{Castelnuovo--Mumford regularity}

Let \( L \) be a nonzero finitely generated graded \( S \)-module. Let \( H_{\mathfrak{m}}^i(L) \) denote the \( i \)-th local cohomology module of \( L \) with support in \( \mathfrak{m} \). Then the  regularity of \( L \) is defined by
\[
\reg(L) = \max \{ j + i \mid H_{\mathfrak{m}}^i(L)_j \neq 0,\ \text{for } i = 0, \ldots, \dim(L),\ j \in \mathbb{Z} \}.
\]

We now recall the following result of Minh, Nam, Phong, Thuy, and Vu \cite{MNPTV}, which provides a method for computing the regularity of a monomial ideal via its associated radicals. To this end, we introduce some notation. For an exponent $\a = (a_1, \ldots, a_n) \in \NN^n$, we denote by $x^\a$ the monomial $x_1^{a_1} \cdots x_n^{a_n}$ in $S$, and we set $|\a| = a_1 + \cdots + a_n.$ The support of $\a$ is defined by
\[
\supp(\a) = \{ i \in [n] \mid a_i \neq 0 \}.
\]
For a positive integer $t$, we define the truncation of $\a$ at level $t$ by $\a^{<t} = (a_1^{<t}, \ldots, a_n^{<t})$, where
\[
a_i^{<t} =
\begin{cases}
a_i, & \text{if } a_i < t,\\
0, & \text{otherwise}.
\end{cases}
\]

\begin{defn}
Let $I$ be a monomial ideal of $S$, and let $\a = (a_1, \ldots, a_n) \in \NN^n$. The \emph{degree complex} of $I$ in degree $\a$ is defined by
\[
\Delta_{\a}(I) = \Delta\bigl(\sqrt{I : x^\a}\bigr).
\]
\end{defn}

By \cite[Lemmas 2.12 and 2.19]{MNPTV}, we have the following result.

\begin{lem}\label{Key0}
Let $I$ be a monomial ideal of $S$. Then
\begin{multline*}
\reg(S/I)
=
\max \bigl\{
|\a| + i
\;\big|\;
\a \in \NN^n,\ i \ge 0,\ 
\widetilde{H}_{i-1}(\lk_{\Delta_{\a}(I)} F; \k) \neq 0 \\
\text{for some } F \in \Delta_{\a}(I)
\text{ such that }
F \cap \supp(\a) = \emptyset
\bigr\}.
\end{multline*}
\end{lem}

\subsection{Hypergraphs and their cover ideals}
We recall some basic notions from hypergraph theory; for further details, see~\cite{BM}. Let $G$ be a simple hypergraph with vertex set $V(G) = \{1, \ldots, n\}$ and edge set $E(G)$. We assume throughout that $G$ has no isolated vertices. Each edge $e$ of $G$ is a subset of $V(G)$, and no edge is properly contained in another; that is, for any two distinct edges $e, f \in E(G)$, neither $e \subsetneq f$ nor $f \subsetneq e$ holds.

A hypergraph $H$ is a \emph{subhypergraph} of $G$ if $V(H) \subseteq V(G)$ and $E(H) \subseteq E(G)$. It is an \emph{induced subhypergraph} of $G$ if a subset $e \subseteq V(H)$ is an edge of $H$ if and only if $e$ is an edge of $G$.

A subset $W \subseteq V(G)$ is called a \emph{cover} of $G$ if $W \cap e \neq \emptyset$ for every edge $e \in E(G)$. It is called a \emph{minimal cover} if no proper subset of $W$ is a cover of $G$. The maximum cardinality of a minimal cover of $G$ is denoted by $\tau_{\max}(G)$.

When $|e| = 2$ for every edge $e$ of $G$, we simply call $G$ a \emph{graph}.

\begin{defn}
Let $G$ be a hypergraph with vertex set $V(G) = \{1, \ldots, n\}$ and edge set $E(G)$. The \emph{edge ideal} and \emph{cover ideal} of $G$, denoted by $I(G)$ and $J(G)$, are defined by
\[
I(G) = (x_e \mid e \in E(G))
\quad \text{and} \quad
J(G) = \bigcap_{e \in E(G)} (x_i \mid i \in e).
\]
\end{defn}

In particular, $J(G)$ is the Alexander dual of $I(G)$. It is well known that
\[
J(G) = (x_W \mid W \text{ is a minimal cover of } G).
\]

\subsection{Symbolic powers of cover ideals and critical subhypergraphs}
\begin{defn}
Let $G$ be a hypergraph on $n$ vertices. The $t$-th symbolic power of $J(G)$ is defined by
\[
J(G)^{(t)}
=
\bigcap_{\{i_1,\ldots,i_s\} \in E(G)}
(x_{i_1},\ldots,x_{i_s})^t.
\]
\end{defn}

\begin{lem}\label{lem_assrad}
Let $G$ be a hypergraph, and let $x^\a$ be a monomial in $S$. Then
\[
\sqrt{J(G)^{(t)} : x^\a} = J(H),
\]
where $H$ is the subhypergraph of $G$ with edge set
\[
E(H)
=
\bigl\{
\{i_1,\ldots,i_s\} \in E(G)
\;\big|\;
a_{i_1} + \cdots + a_{i_s} < t
\bigr\}.
\]
\end{lem}

\begin{proof}
We have
\begin{align*}
\sqrt{J(G)^{(t)} : x^\a}
&=
\sqrt{
\left(
\bigcap_{\{i_1,\ldots,i_s\} \in E(G)}
(x_{i_1},\ldots,x_{i_s})^t
\right)
: x^\a
} \\
&=
\bigcap
\Bigl\{
(x_{i_1},\ldots,x_{i_s})
\;\Big|\;
\{i_1,\ldots,i_s\} \in E(G)
\text{ and }
a_{i_1} + \cdots + a_{i_s} < t
\Bigr\} \\
&=
J(H).
\end{align*}
The conclusion follows.
\end{proof}

We also need the following simple lemma.
\begin{lem}\label{lem_link}
Let $G$ be a hypergraph with vertex set $V(G)$ and edge set $E(G)$. Let $F$ be a face of $\Delta(J(G))$. Let $H$ be the hypergraph with vertex set \( V(H) = V(G) \setminus F \) and edge set \( E(H) = \{e \in E(G) \mid e \subseteq V(G) \setminus F\}. \) Then
\[
\lk_{\Delta(J(G))}(F) = \Delta(J(H)).
\]
\end{lem}
\begin{proof}
By definition, a subset $W \subseteq V(G)$ is a face of $\Delta(J(G))$ if and only if there exists an edge $e \in E(G)$ such that $W \cap e = \emptyset$.

Now let $W \subseteq V(G) \setminus F$. Then $W$ is a face of $\lk_{\Delta(J(G))}(F)$ if and only if $F \cup W$ is a face of $\Delta(J(G))$. Equivalently, there exists an edge $e \in E(G)$ such that
\[
(F \cup W) \cap e = \emptyset.
\]
This is equivalent to the conditions $e \cap F = \emptyset$ and $e \cap W = \emptyset$. The first condition implies that $e \in E(H)$, while the second condition means precisely that $W$ is a face of $\Delta(J(H))$. Therefore,
\[
\lk_{\Delta(J(G))}(F) = \Delta(J(H)),
\]
as desired.
\end{proof}
Note that in Lemma~\ref{lem_link}, $J(H)$ is viewed as an ideal in the polynomial ring whose variables are indexed by the vertex set $V(G) \setminus F$, and the Stanley--Reisner correspondence is taken with respect to this ground set.

\begin{defn}
Let $G$ be a hypergraph with vertex set $V(G) = [n]$, and let $\a \in \mathbb{N}^n$. The $t$-\emph{admissible subhypergraph} of $G$ with respect to $\a$ is the hypergraph $H$ with vertex set
\[
V(H) = \{i \mid a_i < t\}
\]
and edge set
\[
E(H)
=
\bigl\{
\{i_1, \ldots, i_s\} \in E(G)
\;\big|\;
a_{i_1} + \cdots + a_{i_s} < t
\bigr\}.
\]
The pair $(H,\a)$ is called a \emph{$t$-admissible pair} of $G$.
\end{defn}

\begin{defn}
A $t$-admissible pair $(H,\a)$ of $G$ is called a \emph{$t$-critical pair} if $\Delta(J(H))$ is not acyclic. In this case, the hypergraph $H$ is called a \emph{$t$-critical subhypergraph}.
\end{defn}

The following lemma is the main technical result of this section.

\begin{lem}\label{lem_reg_admissible}
Let $G$ be a nonempty hypergraph. Then
\begin{multline*}
\reg\bigl(J(G)^{(t)}\bigr)
=
2 + \max \Bigl\{
\hdim_{\k}\bigl(\Delta(J(H))\bigr) + |\a^{<t}|
\;\Big|\;
(H,\a) \text{ is a $t$-critical pair of } G
\Bigr\}.
\end{multline*}
\end{lem}

\begin{proof}
By Lemma~\ref{Key0} and \cite[Remark 2.13]{MV3}, there exist an exponent $\a \in \NN^n$ and a face $F$ of
\[
\Delta = \Delta_{\a}\bigl(J(G)^{(t)}\bigr)
\]
such that $a_j < t$ for all $j$, $F \cap \supp(\a) = \emptyset$,
$\h_{i-1}\bigl(\lk_{\Delta}(F); \k\bigr) \neq 0$, and
\[
\reg\bigl(J(G)^{(t)}\bigr) = i + 1 + |\a|.
\]

By Lemma~\ref{lem_assrad}, we have $\Delta = \Delta(J(H))$, where $H$ is the subhypergraph of $G$ consisting of all edges
$e = \{i_1, \ldots, i_s\}$ such that $a_{i_1} + \cdots + a_{i_s} < t.$ Let $\b$ be the exponent defined by
\[
x^\b = x^\a \cdot \prod_{j \in F} x_j^t.
\]
Since $F \cap \supp(\a) = \emptyset$, $\b^{<t} = \a$. By Lemmas~\ref{lem_assrad} and~\ref{lem_link}, we obtain
\[
\lk_{\Delta}F = \Delta(J(K)),
\]
where $K$ is the induced subhypergraph of $H$ on the vertex set $V(K) = V(G) \setminus F.$ In particular, $(K,\b)$ is a $t$-critical pair of $G$. Furthermore, since
\[
\h_{i-1}\bigl(\Delta(J(K)); \k\bigr) \neq 0,
\]
we have
\[
\reg\bigl(J(G)^{(t)}\bigr)
=
i + 1 + |\a|
\le
2 + \hdim_{\k}\bigl(\Delta(J(K))\bigr) + |\b^{<t}|.
\]

Conversely, let $(H,\a)$ be a $t$-critical pair of $G$. Set $U = \{i \mid a_i < t\}.$ By definition, $V(H) = U$, and $\Delta(J(H))$ is not acyclic. Let
\[
F = \{i \mid a_i \ge t\}
\quad \text{and} \quad
\b = \a^{<t}.
\]
Let $K$ be the subhypergraph of $G$ such that
\[
J(K) = \sqrt{J(G)^{(t)} : x^\b}.
\]
By definition, $F \cap e = \emptyset$ for every edge $e \in E(H) \subseteq E(K)$. In other words, $F$ is a face of $\Delta(J(K))$. By Lemma~\ref{lem_link},
\[
\lk_{\Delta(J(K))} F = \Delta(J(H)).
\]
Applying Lemma~\ref{Key0}, we obtain
\[
\reg\bigl(J(G)^{(t)}\bigr)
\ge
2 + \hdim_{\k}\bigl(\Delta(J(H))\bigr) + |\b|.
\]
The conclusion follows.
\end{proof}
\begin{rem}
\begin{enumerate}
    \item The idea of eliminating the link in Lemma~\ref{Key0} was initiated by Hoang and Vu \cite[Lemma 2.9]{HV}. In our setting, the resulting formula is somewhat cleaner, with the effect of removing the link being absorbed by the truncation operation. In general, however, this simplification is largely cosmetic. For cover ideals, the formulation is particularly natural, since it involves subhypergraphs of the original hypergraph.
    \item If we adopt the convention that $\hdim_{\k}(\Delta) = -\infty$ whenever $\Delta$ is acyclic, then in Lemma~\ref{lem_reg_admissible} the maximum may be taken over all $t$-admissible pairs. We state the result in terms of critical pairs in order to emphasize the role of $t$-critical subhypergraphs in the study of the regularity of symbolic powers of cover ideals.
    \item Note that every squarefree monomial ideal is the cover ideal of some hypergraph. Hence, the results in this section apply to all symbolic powers of squarefree monomial ideals. Similar approaches have also been used by Hang and Trung \cite{HaT} and by Hien and Trung \cite{HiT} to study the regularity of symbolic powers of squarefree monomial ideals.
\end{enumerate}
\end{rem}

\subsection{Symbolic Polyhedra}

Let $G$ be a hypergraph with vertex set $V(G)$ and edge set $E(G)$, and assume that $G$ has no isolated vertices. The \emph{symbolic polyhedron} of $J(G)$ is defined by
\[
\SP(G)
=
\left\{
x \in \mathbb{R}^n
\;\middle|\;
x_i \ge 0
\text{ and }
\sum_{i \in e} x_i \ge 1
\text{ for all } e \in E(G)
\right\}.
\]

\begin{defn}
The \emph{asymptotic cover degree} of $G$ is defined by
\[
\delta(G)
=
\max \{ |\v| \mid \v \text{ is a vertex of } \SP(G) \}.
\]
\end{defn}

For a homogeneous ideal $I$ of $S$, let $\omega(I)$ denote the maximal degree of a minimal generator of $I$. Dung, Hien, Nguyen, and Trung \cite{DHNT} proved the following fundamental theorem.

\begin{thm}\label{thm_limit_reg}
Let $G$ be a hypergraph without isolated vertices. Then
\[
\lim_{t \to \infty} \frac{\reg\bigl(J(G)^{(t)}\bigr)}{t}
=
\lim_{t \to \infty} \frac{\omega\bigl(J(G)^{(t)}\bigr)}{t}
=
\delta(G).
\]
\end{thm}

The following result, essentially due to Hien and Trung \cite[Theorem 2.2]{HiT}, provides a sharp upper bound for $|\a^{<t}|$ when $(H,\a)$ is a $t$-critical pair of $G$. It was proved in the language of critical exponents by Minh and Vu \cite[Theorem 2.18]{MV3}. We give an alternative proof using our terminology.

\begin{thm}\label{thm_degree_bound}
Let $G$ be a nonempty hypergraph, and let $(H,\a)$ be a $t$-critical pair of $G$. Then
\[
|\a^{<t}| \le \delta(G)(t-1).
\]
\end{thm}

\begin{proof}
When $t = 1$, the claim is immediate from the definition. Assume now that $t \ge 2$. We claim that $(H, s\a)$ is an $(s(t-1)+1)$-critical pair of $G$ for every positive integer $s$. Indeed, for every edge $e \in E(H)$, we have
\[
\sum_{i \in e} sa_i
=
s\sum_{i \in e} a_i
\le
s(t-1)
<
s(t-1)+1.
\]
For every edge $e \in E(G) \setminus E(H)$, we have
\[
\sum_{i \in e} sa_i
=
s\sum_{i \in e} a_i
\ge
st
\ge
s(t-1)+1.
\]
Hence, by definition, $(H,s\a)$ is an $(s(t-1)+1)$-critical pair of $G$. By Lemma~\ref{lem_reg_admissible}, we obtain
\[
\reg\bigl(J(G)^{(s(t-1)+1)}\bigr)
>
s |\a^{<t}|.
\]
Dividing both sides by $s(t-1)+1$ and letting $s \to \infty$, we deduce that
\[
\frac{|\a^{<t}|}{t-1}
\le
\lim_{s \to \infty}
\frac{\reg\bigl(J(G)^{(s(t-1)+1)}\bigr)}{s(t-1)+1}
=
\delta(G),
\]
where the equality follows from Theorem~\ref{thm_limit_reg}. The desired inequality follows.
\end{proof}

The following lemma allows us, in many situations, to reduce to the case of connected hypergraphs.

\begin{lem}\label{lem_disjoint}
Let $G_1$ and $G_2$ be hypergraphs with vertex sets $V(G_1)$ and $V(G_2)$, respectively, such that
\(V(G_1) \cap V(G_2) = \emptyset.\) Assume that $G = G_1 \cup G_2$. Then
\[
\tau_{\max}(G) = \tau_{\max}(G_1) + \tau_{\max}(G_2)
\]
and
\[
\hdim_{\k}\bigl(\Delta(J(G))\bigr)
=
\hdim_{\k}\bigl(\Delta(J(G_1))\bigr)
+
\hdim_{\k}\bigl(\Delta(J(G_2))\bigr)
+
2.
\]
\end{lem}

\begin{proof}
The statement concerning $\tau_{\max}$ follows immediately from the definition.

For simplicity, let
\[
\Delta = \Delta(J(G)), \qquad
\Delta_1 = \Delta(J(G_1)), \qquad
\Delta_2 = \Delta(J(G_2)).
\]
Then
\(
\Delta
=
\Delta_1 * V(G_2)
\;\cup\;
V(G_1) * \Delta_2.
\) Both $\Delta_1 * V(G_2)$ and $V(G_1) * \Delta_2$ are cones, and hence are acyclic. Moreover,
\[
\Delta_1 * V(G_2)
\;\cap\;
V(G_1) * \Delta_2
=
\Delta_1 * \Delta_2.
\]
It follows from the Mayer--Vietoris sequence that
\[
\hdim_{\k}(\Delta)
=
\hdim_{\k}(\Delta_1 * \Delta_2) + 1.
\]
By the K\"unneth formula for joins,
\[
\hdim_{\k}(\Delta_1 * \Delta_2)
=
\hdim_{\k}(\Delta_1)
+
\hdim_{\k}(\Delta_2)
+
1.
\]
The conclusion follows.
\end{proof}

\section{Applications}\label{sec_3}
In this section, we apply the results of the previous section to determine the regularity of symbolic powers of cover ideals of bipartite unicyclic graphs. We begin by recalling some basic definitions.

A cycle $C_n$ on $n$ vertices is the graph with vertex set
\(V(C_n) = \{1, \ldots, n\}
\) and edge set
\(E(C_n)
=
\{\{1,2\}, \ldots, \{n-1,n\}, \{1,n\}\}.
\)
The homology of the independence complexes of cycles was determined by Kozlov \cite{Ko}. In particular, we have the following result.

\begin{lem}\label{lem_hdim_cycle}
Let $C_n$ denote a cycle on $n \ge 3$ vertices. Then
\begin{align*}
    \tau_{\max}(C_n) &= \left\lfloor \frac{2n}{3} \right\rfloor,\\
    \hdim_{\k}\bigl(\Delta(J(C_n))\bigr)
    &= \left\lfloor \frac{2n-4}{3} \right\rfloor.
\end{align*}
In particular,
\[
\tau_{\max}(C_n) - \hdim_{\k}\bigl(\Delta(J(C_n))\bigr)
=
\begin{cases}
1 & \text{if } n \equiv 1 \pmod{3},\\
2 & \text{otherwise}.
\end{cases}
\]
\end{lem}

A graph $G$ is called \emph{unicyclic} if it contains exactly one cycle. In particular, every connected unicyclic graph can be obtained by attaching a tree to each vertex of a cycle.

It turns out that the graph simplification introduced in \cite{PV}, which consists of truncating branches at branching points, is particularly useful for analyzing critical subgraphs of unicyclic graphs. The reduction lemmas in this section are motivated by these simplification moves. We now introduce some terminology. Let $u$ be a vertex of a graph $G$. The \emph{neighborhood} of $u$ in $G$, denoted by $N_G(u)$, is the set of all vertices $v$ such that $\{u,v\}$ is an edge of $G$. The \emph{degree} of $u$, denoted by $\deg_G(u)$, is the cardinality of $N_G(u)$. A vertex $u$ is called a \emph{branching point} if $\deg_G(u) > 2$, and a \emph{leaf} if $\deg_G(u) = 1$.

Let $G$ be a simple graph, and let $u$ be a branching point of $G$. A path that starts at $u$ and ends at a leaf is called a \emph{pure branch} if every vertex on the path, except for $u$, has degree at most $2$ in $G$.

We define the following reduction moves:

\begin{enumerate}
    \item If a vertex $u$ has $s \ge 2$ leaves attached to it, remove all but one of these leaves.
    \item If a vertex $u$ has $s \ge 2$ pure branches of length $2$ attached to it, remove all but one of these branches.
    \item If a vertex $u$ has a pure branch of length at least $3$ attached to it, shorten the branch by $3$ by removing the three vertices closest to the leaf, including the leaf itself.
\end{enumerate}

\begin{defn}
Let $G$ be a simple graph. If a graph $H$ can be obtained from $G$ by a sequence of the reduction moves above, then $H$ is called a \emph{reduction} of $G$. We also say that $G$ is \emph{reducible} to $H$.
\end{defn}

\begin{exam}
The graph in Figure~\ref{fig:reducible_c4} is reducible to $C_4$ by first removing a branch of length $2$ and then deleting the three vertices closest to the leaf in a pure branch of length $3$.

\begin{figure}[h]
\centering
\begin{tikzpicture}[scale=1.2, every node/.style={circle, fill=black, inner sep=2pt}]

% Vertices of the diamond
\node (1) at (0,0) {};
\node (2) at (1,1) {};
\node (3) at (2,0) {};
\node (4) at (1,-1) {};

% Central vertex
\node (5) at (3,0) {};

% Upper branch
\node (6) at (4,1) {};
\node (7) at (5,1) {};

% Lower branch
\node (8) at (4,-1) {};
\node (9) at (5,-1) {};

% Diamond edges
\draw (1) -- (2);
\draw (2) -- (3);
\draw (3) -- (4);
\draw (4) -- (1);

% Connection to the central vertex
\draw (3) -- (5);

% Upper branch
\draw (5) -- (6);
\draw (6) -- (7);

% Lower branch
\draw (5) -- (8);
\draw (8) -- (9);

\end{tikzpicture}
\caption{A graph reducible to $C_4$.}
\label{fig:reducible_c4}
\end{figure}
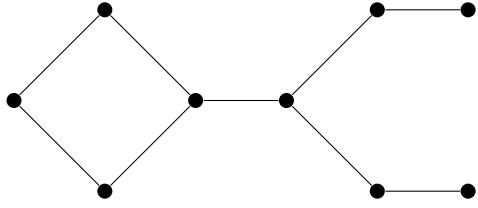

\end{exam}

\begin{exam}
The graph in Figure~\ref{fig:nreducible_c4} is not reducible to $C_4$. At one of its branch vertices, it has both a branch of length $1$ and a branch of length $2$. Hence, the simplicial complex associated with its cover ideal is acyclic, so the graph cannot appear as a critical candidate.

\begin{figure}[h]
\centering
\begin{tikzpicture}[scale=1.2,
    every node/.style={circle, fill=black, inner sep=2pt}]

% Diamond C4: vertices 1,2,3,4
\node (1) at (0,0) {};
\node (2) at (1,1) {};
\node (3) at (2,0) {};
\node (4) at (1,-1) {};

% Main path: 3-5-6-7-8
\node (5) at (3,0) {};
\node (6) at (4,0) {};
\node (7) at (5,0) {};
\node (8) at (6,0) {};

% Upper and lower leaves from vertex 6
\node (13) at (4,1) {};
\node (14) at (4,-1) {};

% Diamond edges
\draw (1) -- (2);
\draw (2) -- (3);
\draw (3) -- (4);
\draw (4) -- (1);

% Main path edges
\draw (3) -- (5);
\draw (5) -- (6);
\draw (6) -- (7);
\draw (7) -- (8);

% Leaves at vertex 6
\draw (6) -- (13);
\draw (6) -- (14);

\end{tikzpicture}
\caption{A graph not reducible to $C_4$.}
\label{fig:nreducible_c4}
\end{figure}
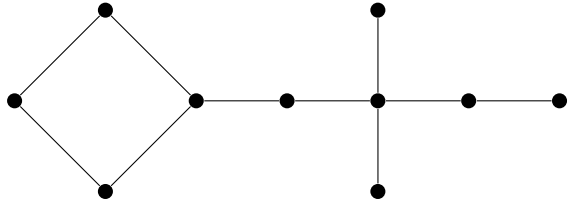

\end{exam}

We now outline the main ideas in the proof of Theorem~\ref{thm_reg_unicycli} and Theorem~\ref{thm_reg_pow_uni}.
\begin{enumerate}
    \item First, we consider the case where $G = C_n$ separately.

    \item Suppose that $G$ has a leaf. Then either one of the reduction steps described above can be applied. In each case, we reduce to an induced subgraph $G_1$ of $G$ such that every critical subgraph of $G$ gives rise to a critical subgraph of $G_1$, with the homological dimension decreasing by the appropriate amount. This is established in the reduction lemmas that follow.

    \item If none of the reduction steps can be applied, then $\Delta(J(G))$ is acyclic. In other words, any critical subgraph of $G$ must arise from a proper subgraph of $G$, and the conclusion follows by induction. Thus, either $G$ is reducible to a path of length at most $2$, which is straightforward to verify, or to a cycle, which is handled in the first step, or $\Delta(J(G))$ is acyclic and $G$ contains a proper subgraph to which one of the reduction steps can be applied.
    \item When considering powers, one only needs to keep track of the degree $|\a|$ of the exponent vector $\a$ in a critical pair $(H,\a)$. The decrease in $\tau_{\max}$ between $G$ and its reductions accounts for the corresponding change in this degree.
\end{enumerate}

\begin{lem}\label{lem_reduction_1}
Let $G$ be a simple graph, and let $u$ be a vertex of $G$. Assume that $u$ has at least two leaf neighbors $v_1$ and $v_2$. Let $H = G \setminus v_1$. Then:
\begin{enumerate}
    \item
    \(    \tau_{\max}(H) \le \tau_{\max}(G) - 1;    \)
    
    \item
    \(    \hdim_{\k}\bigl(\Delta(J(H))\bigr)
    =
    \hdim_{\k}\bigl(\Delta(J(G))\bigr) - 1.
    \)
\end{enumerate}
\end{lem}

\begin{proof}
\textbf{(1)} Let $C$ be a minimal cover of $H$ with $|C| = \tau_{\max}(H)$. If $u \notin C$, then $C \cup \{v_1\}$ is a minimal cover of $G$. Hence,
\[
\tau_{\max}(G) \ge |C| + 1 = \tau_{\max}(H) + 1,
\]
which proves the claim. Now assume that $u \in C$. Let
\[
L = \{v \in N_H(u) \mid v \text{ is a leaf}\}
 \text{ and } 
A = \{v \in N_H(u) \mid v \notin C \text{ and } v \text{ is not a leaf}\}.
\]
Then
\[
C' = (C \setminus \{u\}) \cup L \cup A
\]
is a minimal cover of $H$, and \(|C'| \ge |C|.\) Since $u \notin C'$, we are reduced to the previous case. Therefore,
\[
\tau_{\max}(H) \le \tau_{\max}(G) - 1.
\]

\medskip

\textbf{(2)} Let
\[
\Delta = \Delta(J(G)),
\qquad
F = V(G) \setminus \{u,v_1\},
\qquad
\Gamma = \Delta \setminus F.
\]
Then $\Gamma$ is a cone with apex $v_1$, and hence is acyclic. By the Mayer--Vietoris sequence, we obtain
\begin{equation}\label{eq1}
\hdim_{\k}(\Delta)
=
\hdim_{\k}(F \cap \Gamma) + 1.
\end{equation}
Now, $F \cap \Gamma$ is the union of the facet
\[
F_1 = V(G) \setminus \{u,v_1,v_2\},
\]
and a simplicial complex $\Gamma_1$ whose facets are of the form
\[
V(G) \setminus \bigl(\{u,v_1\} \cup e\bigr),
\]
where $e$ is an edge of $H$ other than $\{u,v_2\}$. Note that $\Gamma_1$ is a cone with apex $v_2$. Hence,
\begin{equation}\label{eq2}
\hdim_\k (F \cap \Gamma) = \hdim_{\k}(F_1 \cup \Gamma_1)
=
\hdim_{\k}(F_1 \cap \Gamma_1) + 1.
\end{equation}
Now, $F_1 \cap \Gamma_1$ is the simplicial complex $\Sigma$ consisting of the facets
\[
V(G) \setminus \bigl(\{u,v_1,v_2\} \cup e\bigr),
\]
where $e$ is an edge of $H$ other than $\{u,v_2\}$.

On the other hand, decompose $\Delta(J(H))$ as \( \Delta(J(H)) = F_2 \cup \Lambda,\) where
\[
F_2 = V(H) \setminus \{u,v_2\}
     = V(G) \setminus \{u,v_1,v_2\},
\]
and $\Lambda$ is the subcomplex generated by all facets other than $F_2$. Then
\[
F_2 \cap \Lambda = \Sigma.
\]
Since $\Lambda$ is also a cone with apex $v_2$, it follows that
\[
\hdim_{\k}\bigl(\Delta(J(H))\bigr)
=
\hdim_{\k}(F_2 \cap \Lambda) + 1
=
\hdim_{\k}(\Sigma) + 1.
\]
Combining this with Eq.~\eqref{eq1} and Eq.~\eqref{eq2}, we conclude that
\[
\hdim_{\k}\bigl(\Delta(J(H))\bigr)
=
\hdim_{\k}\bigl(\Delta(J(G))\bigr) - 1.
\]
This completes the proof.
\end{proof}

\begin{lem}\label{lem_reduction_2}
Let $G$ be a simple graph, and let $u$ be a vertex of $G$. Assume that $u$ has a pure branch
\(u, v_1, \ldots, v_k,\) where $v_k$ is a leaf and $k \ge 3$. Let
\(H = G \setminus \{v_k, v_{k-1}, v_{k-2}\}. \) Then:
\begin{enumerate}
    \item     \(    \tau_{\max}(H) \le \tau_{\max}(G) - 2;    \)

    \item     \(    \hdim_{\k}\bigl(\Delta(J(H))\bigr)
    =
    \hdim_{\k}\bigl(\Delta(J(G))\bigr) - 2.
    \)
\end{enumerate}
\end{lem}

\begin{proof}
\textbf{(1)} Let $C$ be a minimal vertex cover of $H$ with
\( |C| = \tau_{\max}(H).\) Then
\( C \cup \{v_{k-2}, v_k\} \) is a minimal vertex cover of $G$. Hence,
\[
\tau_{\max}(G) \ge \tau_{\max}(H) + 2,
\]
which proves part~(1).

\medskip

\textbf{(2)} Let
\[
\Delta = \Delta(J(G)),
\qquad
F_1 = V(G) \setminus \{v_{k-1}, v_k\}, \qquad
\Gamma = \Delta \setminus F_1.
\]
Then $\Gamma$ is a cone with apex $v_k$. Therefore,
\[
\hdim_{\k}(\Delta)
=
\hdim_{\k}(F_1 \cap \Gamma) + 1.
\]

Now, $F_1 \cap \Gamma$ is the union of the facet
\[
G_1 = V(G) \setminus \{v_k, v_{k-1}, v_{k-2}\}
\]
and a simplicial complex $\Gamma_1$ whose facets are of the form
\[
V(G) \setminus \bigl(\{v_k, v_{k-1}\} \cup e\bigr),
\]
where $e$ is an edge of $H$. The complex $\Gamma_1$ is a cone with apex $v_{k-2}$, and
\[
G_1 \cap \Gamma_1 = \Delta(J(H)).
\]
It follows that
\[
\hdim_{\k}(F_1 \cap \Gamma)
=
\hdim_{\k}\bigl(\Delta(J(H))\bigr) + 1.
\]
Consequently,
\[
\hdim_{\k}\bigl(\Delta(J(G))\bigr)
=
\hdim_{\k}\bigl(\Delta(J(H))\bigr) + 2,
\]
which proves part~(2).
\end{proof}

\begin{lem}\label{lem_reduction_3}
Let $G$ be a simple graph, and let $u$ be a vertex of $G$. Assume that $u$ has a leaf $v_1$ and a path \( u, v_2, v_3\), where $v_3$ is a leaf of $G$. Then $\Delta(J(G))$ is acyclic.
\end{lem}

\begin{proof}
Let \( \Delta = \Delta(J(G)), \) and define the facets
\[
F_1 = V(G) \setminus \{v_2, v_3\}, \qquad
F_2 = V(G) \setminus \{u, v_2\}, \qquad
F_3 = V(G) \setminus \{u, v_1\}.
\]
Let $\Gamma$ be the subcomplex of $\Delta$ obtained by removing the facets $F_1$, $F_2$, and $F_3$. Then
\[
F_2 \cup F_3 \cup \Gamma
\]
is a cone with apex $v_3$. Therefore,
\[
\hdim_{\k}(\Delta)
=
\hdim_{\k}\bigl(F_1 \cap (F_2 \cup F_3 \cup \Gamma)\bigr) + 1.
\]
Observe that
\[
F_1 \cap F_3
=
V(G) \setminus \{u, v_1, v_2, v_3\}
\subseteq
F_1 \cap F_2
=
V(G) \setminus \{u, v_2, v_3\}.
\]
Hence,
\[
F_1 \cap (F_2 \cup F_3 \cup \Gamma)
=
F_1 \cap (F_2 \cup \Gamma).
\]
Finally, $F_1 \cap (F_2 \cup \Gamma)$ is a cone with apex $v_1$, and is therefore acyclic. It follows that $\Delta(J(G))$ is acyclic.
\end{proof}

\begin{lem}\label{lem_reduction_4}
Let $G$ be a simple graph, and let $u$ be a vertex of $G$. Assume that $u$ has at least two pure branches of length $2$ attached to it, namely $u,v_1,v_2$ and $u,v_3,v_4$. Let $H$ be the induced subgraph of $G$ on $V(G) \setminus \{v_1,v_2\}$. Then:
\begin{enumerate}
    \item
    \(    \tau_{\max}(H) \le \tau_{\max}(G) - 1;
    \)

    \item
    \(
    \hdim_{\k}\bigl(\Delta(J(H))\bigr)
    =
    \hdim_{\k}\bigl(\Delta(J(G))\bigr) - 1.
    \)
\end{enumerate}
\end{lem}

\begin{proof}
\textbf{(1)} Let $C$ be a minimal vertex cover of $H$ with \( |C| = \tau_{\max}(H).\) Then either $C \cup \{v_1\}$ or $C \cup \{v_2\}$ is a minimal vertex cover of $G$. Hence,
\[
\tau_{\max}(G) \ge \tau_{\max}(H) + 1.
\]

\medskip

\textbf{(2)} Let
\[
\Delta = \Delta(J(G)),
\qquad
F = V(G) \setminus \{v_1,v_2\},
\qquad
\Gamma = \Delta \setminus F.
\]
Then $\Gamma$ is a cone with apex $v_2$. Therefore,
\[
\hdim_{\k}(\Delta)
=
\hdim_{\k}(F \cap \Gamma) + 1.
\]
The complex $F \cap \Gamma$ consists of the facet
\( F_1 = V(G) \setminus \{u,v_1,v_2\},
\) together with a simplicial complex $\Gamma_1$ whose facets are of the form
\[
V(G) \setminus \bigl(\{v_1,v_2\} \cup e\bigr),
\]
where $e$ ranges over the edges of the induced subgraph of $G$ on \(V(G) \setminus \{u,v_1,v_2\}.\)
The complex $\Gamma_1$ is a cone with apex $u$. Hence,
\[
\hdim_{\k}(F \cap \Gamma)
=
\hdim_{\k}(F_1 \cap \Gamma_1) + 1.
\]
Now, $F_1 \cap \Gamma_1$ consists of the facet
\[
F_2 = V(G) \setminus \{u,v_1,v_2,v_3,v_4\},
\]
together with a simplicial complex $\Gamma_2$ whose facets are of the form
\[
V(G) \setminus \bigl(\{u,v_1,v_2\} \cup e\bigr),
\]
where $e$ ranges over the edges of the induced subgraph of $G$ on \(V(G) \setminus \{u,v_1,v_2,v_3,v_4\}.\) The complex $\Gamma_2$ is a cone with apex $v_3$. Therefore,
\[
\hdim_{\k}(F_1 \cap \Gamma_1)
=
\hdim_{\k}(F_2 \cap \Gamma_2) + 1.
\]
Finally, $F_2 \cap \Gamma_2$ is the simplicial complex $\Sigma$ whose facets are of the form
\[
V(G) \setminus \bigl(\{u,v_1,v_2,v_3,v_4\} \cup e\bigr),
\]
where $e$ ranges over the edges of the induced subgraph of $G$ on
\[
V(G) \setminus \{u,v_1,v_2,v_3,v_4\}.
\]

Now decompose $\Delta(J(H))$ using the facet \(V(H) \setminus \{v_3,v_4\}.\) Applying the same procedure twice, we arrive at the same simplicial complex $\Sigma$. Hence,
\[
\hdim_{\k}\bigl(\Delta(J(G))\bigr)
=
\hdim_{\k}\bigl(\Delta(J(H))\bigr) + 1.
\]
This completes the proof.
\end{proof}

\begin{rem}
Note that in these first four reduction lemmas---Lemma~\ref{lem_reduction_1}, Lemma~\ref{lem_reduction_2}, Lemma~\ref{lem_reduction_3}, and Lemma~\ref{lem_reduction_4}---we do not assume that $G$ is unicyclic. The graph $G$ may be arbitrary, provided that it contains a vertex satisfying the corresponding configuration in the relevant lemma. In particular, the homology of $\Delta(J(G))$ may itself be quite complicated.
\end{rem}

When $G$ contains a configuration as in Lemma~\ref{lem_reduction_3}, the simplicial complex $\Delta(J(G))$ is acyclic. In this case, we perform a reduction that may produce a non-acyclic critical subgraph of $G$ while decreasing $\tau_{\max}(G)$ by one.

\begin{lem}\label{lem_type_1_2}
Let $G$ be a simple graph, and let $u$ be a branching vertex of $G$ of degree $3$. Assume that $u$ is incident to a branch of length $1$, namely $u,v_1$, and to a pure branch of length $2$, namely $u,v_2,v_3$, so that \(N_G(u) = \{v_1,v_2,w\}.\) Let $H = G \setminus v_1$. Then
\[
\tau_{\max}(H) \le \tau_{\max}(G) - 1.
\]
\end{lem}

\begin{proof}
Let $C$ be a minimal vertex cover of $H$ with $|C| = \tau_{\max}(H)$. If $u \notin C$, then $C \cup \{v_1\}$ is a minimal vertex cover of $G$, and we are done.

Now assume that $u \in C$. Then at most one of $w$ and $v_2$ belongs to $C$.

\begin{enumerate}
    \item If $w, v_2 \notin C$, then $v_3 \in C$, and we set \(C' = (C \setminus \{u,v_3\}) \cup \{w,v_2\}.\)
    \item If $w \notin C$ and $v_2 \in C$, then we set
\(
C' = (C \setminus \{u\}) \cup \{w\}.
\)
\item If $w \in C$ and $v_2 \notin C$, then we set
\(
C' = (C \setminus \{u\}) \cup \{v_2\}.
\)
\end{enumerate}
In each case, $C'$ is a minimal vertex cover of $H$ with $|C'| = |C|$, and $C' \cup \{v_1\}$ is a minimal vertex cover of $G$. The conclusion follows.
\end{proof}

In the remaining two reduction lemmas, we assume that $G$ is unicyclic and that at most one pure branch is attached to any vertex on the unique cycle of $G$.
\begin{lem}\label{lem_reduction_5}
Assume that $G$ is a unicyclic graph whose unique cycle is
\(u_1, \ldots, u_n,\) where $n \ge 4$. Suppose that $u_1$ has a unique leaf $v_1$ attached to it, and that
\[
N_G(u_2) = \{u_1,u_3\}
\quad \text{and} \quad
N_G(u_n) = \{u_1,u_{n-1}\}.
\]
Let $H$ be the induced subgraph of $G$ on
on \(V(G) \setminus \{u_1, v_1, u_2, u_n\}.\) Then:
\begin{enumerate}
    \item     \(     \tau_{\max}(H) \le \tau_{\max}(G) - 3;    \) 
    \item     \(     \hdim_{\k}\bigl(\Delta(J(H))\bigr)
    =
    \hdim_{\k}\bigl(\Delta(J(G))\bigr) - 3.
    \)
\end{enumerate}
\end{lem}
\begin{proof}
Let $C$ be a minimal cover of $H$. Then \( C \cup \{v_1, u_2, u_n\} \) is a minimal vertex cover of $G$. Hence,
\[
\tau_{\max}(H) \le \tau_{\max}(G) - 3.
\]

The proof of part~(2) is similar to that of Lemma~\ref{lem_reduction_4}. One first separates the facet corresponding to the edge $\{u_1, v_1\}$, then takes intersections and applies the Mayer--Vietoris sequence repeatedly. This yields
\[
\hdim_{\k}\bigl(\Delta(J(H))\bigr)
=
\hdim_{\k}\bigl(\Delta(J(G))\bigr) - 3.
\]
The conclusion follows.
\end{proof}

\begin{lem}\label{lem_reduction_6}
Assume that $G$ is a unicyclic graph whose unique cycle is \(u_1, \ldots, u_n,
\) where $n \ge 3$. Suppose that
\[
N_G(u_1) = \{u_2,u_n,v_1\}
\quad \text{and} \quad
N_G(v_1) = \{u_1,v_2\},
\]
where $v_2$ is a leaf of $G$. Let $H$ be the induced subgraph of $G$ on
\(V(G) \setminus \{u_1,v_1,v_2\}.\) Then:
\begin{enumerate}
    \item
    \(
    \tau_{\max}(H) \le \tau_{\max}(G) - 2;
    \)
    \item
    \(
    \hdim_{\k}\bigl(\Delta(J(H))\bigr)
    =
    \hdim_{\k}\bigl(\Delta(J(G))\bigr) - 2.
    \)
\end{enumerate}
\end{lem}
\begin{proof}
Let $C$ be a minimal vertex cover of $H$. Then
\(
C \cup \{u_1,v_2\}
\) is a minimal vertex cover of $G$. This proves part~(1).

To prove part~(2), we first separate the facet
\(V(G) \setminus \{v_1,v_2\}.
\) Then, by taking intersections and applying the Mayer--Vietoris sequence twice, we obtain the desired equality.
\end{proof}

\begin{proof}[Proof of Theorem \ref{thm_reg_unicycli}]
If $G$ is a forest, or if its unique cycle has length $3$ or $5$, then the conclusion follows from a result of Woodroofe \cite{W}. The case in which $G = C_n$ follows from a result of Jacques \cite{J}.

By Lemma~\ref{lem_disjoint}, we may assume that $G$ is connected, that its unique cycle is $C_n$ with $n \ge 6$, and that $G$ has a leaf $v$. For degree reasons, \( \reg\bigl(J(G)\bigr) \ge \tau_{\max}(G).
\) We now prove by induction on the number of vertices of $G$ that
\[
\reg\bigl(J(G)\bigr) \le \tau_{\max}(G).
\]
By Lemma~\ref{lem_reg_admissible}, it suffices to show that for every critical admissible subgraph $H$ of $G$, we have
\[
\hdim_{\k}\bigl(\Delta(J(H))\bigr)
\le
\tau_{\max}(G) - 2.
\]

First, suppose that $G$ has a vertex $u$ with at least two leaves $v_1$ and $v_2$. If one of these leaves does not belong to $H$, say $v_1 \notin V(H)$, let
\[
G_1 = G \setminus v_1.
\]
Then $H$ is a critical admissible subgraph of $G_1$. By the induction hypothesis,
\[
\hdim_{\k}\bigl(\Delta(J(H))\bigr)
\le
\tau_{\max}(G_1) - 2
\le
\tau_{\max}(G) - 2.
\]

Now suppose that both $v_1$ and $v_2$ belong to $H$. Let
\[
H_1 = H \setminus v_1.
\]
Then $H_1$ is a critical admissible subgraph of $G_1$. By Lemma~\ref{lem_reduction_1} and the induction hypothesis, we obtain
\[
\hdim_{\k}\bigl(\Delta(J(H))\bigr)
=
\hdim_{\k}\bigl(\Delta(J(H_1))\bigr) + 1
\le
\tau_{\max}(G_1) - 1
\le
\tau_{\max}(G) - 2.
\]

Similarly, applying Lemmas~\ref{lem_reduction_2}, \ref{lem_reduction_3}, \ref{lem_reduction_4}, and \ref{lem_reduction_5}, together with the induction hypothesis, yields the desired conclusion.
\end{proof}

\begin{rem}
Hang and Vu \cite{HaV} remarked that the projective dimension of the edge ideal of a unicyclic graph can also be computed recursively using the method of Betti splittings described in their paper. By a result of Terai \cite{T}, we have
\(\pd(S/I(G)) = \reg(J(G)).\) Hence, Theorem~\ref{thm_reg_unicycli} yields a simple formula for the projective dimension of the edge ideal of a unicyclic graph. Note that, by a result of Woodroofe \cite{W}, the cover ideals of unicyclic graphs are not componentwise linear in general.
\end{rem}

\begin{lem}\label{lem_pow_cycles}
Let $G = C_n$ be a cycle on $n$ vertices. Then, for all $t \ge 2$, we have
\[
\reg\bigl(J(G)^{(t)}\bigr)
=
\begin{cases}
2t + 1 & \text{if } n = 4,\\
\left\lfloor \dfrac{2n}{3} \right\rfloor t & \text{otherwise}.
\end{cases}
\]
\end{lem}

\begin{proof}
Note that
\(\tau_{\max}(G) = \left\lfloor \frac{2n}{3} \right\rfloor.
\) If $n = 4$, then
\(J(G) = (x_1x_3, x_2x_4),
\) which is a complete intersection. Hence,
\(\reg\bigl(J(G)^{(t)}\bigr) = 2t + 1
\) for all $t \ge 1$.

Now assume that $n \neq 4$. By degree considerations and Lemma~\ref{lem_reg_admissible}, it suffices to show that if $(H,\a)$ is a critical pair for $G$, then
\[
\hdim_{\k}\bigl(\Delta(J(H))\bigr) + |\a^{<t}|
\le
t \left\lfloor \frac{2n}{3} \right\rfloor - 2.
\]

First, assume that $H$ is a proper subgraph of $C_n$. We may assume that $\{1,n\} \notin E(H)$. Let $G_1$ be the path obtained from $C_n$ by deleting the edge $\{1,n\}$. Then, by definition, $(H,\a)$ is a critical pair for $G_1$. By Lemma~\ref{lem_reg_admissible} and the case of paths (which follows from \cite{BKK}, or can be proved by induction as in the proof of Theorem~\ref{thm_reg_pow_uni}), we obtain
\[
|\a| + \hdim_{\k}\bigl(\Delta(J(H))\bigr)
\le
t \tau_{\max}(G_1) - 2
=
t \tau_{\max}(G) - 2.
\]

Now assume that $H = G$. In other words,
\[
a_i + a_{i+1} \le t - 1
\qquad
\text{for all } i = 1, \ldots, n.
\]
Hence,
\[
|\a|
\le
\left\lceil \frac{n}{2} \right\rceil (t - 1).
\]
Therefore,
\[
|\a| + \hdim_{\k}\bigl(\Delta(J(G))\bigr)
\le
\left\lceil \frac{n}{2} \right\rceil (t - 1)
+
\hdim_{\k}\bigl(\Delta(J(G))\bigr).
\]

When $n = 3$ or $5$, we have
\[
\left\lceil \frac{n}{2} \right\rceil = \tau_{\max}(G) \quad \text{and} \quad 
\hdim_{\k}\bigl(\Delta(J(G))\bigr) = \tau_{\max}(G) - 2,
\]
and therefore
\[
\left\lceil \frac{n}{2} \right\rceil (t-1)
+
\hdim_{\k}\bigl(\Delta(J(G))\bigr)
=
t\tau_{\max}(G) - 2.
\]

If $n \ge 6$, then
\[
\left\lceil \frac{n}{2} \right\rceil < \tau_{\max}(G) \quad \text{ and } \quad 
\hdim_{\k}\bigl(\Delta(J(G))\bigr)
\le
\tau_{\max}(G) - 1.
\]
Consequently,
\[
\left\lceil \frac{n}{2} \right\rceil (t-1)
+
\hdim_{\k}\bigl(\Delta(J(G))\bigr)
\le
t\tau_{\max}(G) - 2,
\]
which yields the desired inequality.
\end{proof}

\begin{thm}\label{thm_reg_pow_uni}
Let $G$ be a unicyclic graph whose unique cycle is $C_n$. Assume that $\delta(G) = \tau_{\max}(G)$. Then for all $t \ge 2$,
\[
\reg \bigl(J(G)^{(t)}\bigr)
=
\begin{cases}
t \tau_{\max}(G) + 1 & \text{if } n = 4 \text{ and } G \text{ is reducible to } C_4,\\
t \tau_{\max}(G) & \text{otherwise}.
\end{cases}
\]
\end{thm}

\begin{proof}
We prove the case $n \neq 4$, and leave the case $n = 4$ to the interested reader. The proof in this case follows a similar line of argument, with additional attention to the homological dimension.

By degree reasons, it suffices to show that
\[
\reg\bigl(J(G)^{(t)}\bigr) \le t \tau_{\max}(G).
\]
We proceed by induction on the number of vertices of $G$. The case $n \le 2$ is clear.

By Lemma~\ref{lem_reg_admissible}, it suffices to show that if $(H,\a)$ is a $t$-critical pair of $G$, then
\[
\hdim_{\k}\bigl(\Delta(J(H))\bigr) + |\a^{<t}| \le t \tau_{\max}(G) - 2.
\]

By Lemma~\ref{lem_pow_cycles}, we now assume that $G$ has a leaf. We perform a sequence of reduction moves. We present one case in detail; the remaining cases are handled similarly.

Suppose that $G$ has a vertex $u$ with at least two leaves, say $v_1$ and $v_2$. First, assume that one of these leaves does not belong to $H$, say $v_1 \notin V(H)$. Let $G_1 = G \setminus v_1$. Then $(H,\a')$, where $\a'$ is the restriction of $\a$ to $V(G_1)$, is a $t$-critical pair of $G_1$. Hence, by the induction hypothesis,
\[
|\a'^{<t}| + \hdim_\k\bigl(\Delta(J(H))\bigr)
\le
t \tau_{\max}(G_1) - 2.
\]
By Lemma~\ref{lem_reduction_1}, we obtain
\[
\begin{aligned}
|\a^{<t}| + \hdim_\k\bigl(\Delta(J(H))\bigr)
&\le
|\a'^{<t}| + (t-1) + \hdim_\k\bigl(\Delta(J(H))\bigr) \\
&\le
t \tau_{\max}(G_1) - 2 + (t-1) \\
&\le
t \tau_{\max}(G) - 2.
\end{aligned}
\]
Now assume that all leaves attached to $u$ remain in $H$. Let $G_1 = G \setminus v_1$ as above, and set $H_1 = H \setminus v_1$. Then $(H_1,\a')$ is a $t$-critical pair of $G_1$. By the induction hypothesis and Lemma~\ref{lem_reduction_1}, we deduce that
\[
\begin{aligned}
|\a^{<t}| + \hdim_\k\bigl(\Delta(J(H))\bigr)
&\le
|\a'^{<t}| + (t-1)
+ \hdim_\k\bigl(\Delta(J(H_1))\bigr) + 1 \\
&\le
t \tau_{\max}(G_1) - 2 + (t-1) + 1 \\
&\le
t \tau_{\max}(G) - 2.
\end{aligned}
\]
Thus, we may assume that every vertex of $G$ has at most one leaf attached to it. 

We then apply the reduction steps described in Lemmas~\ref{lem_reduction_2}, \ref{lem_reduction_3}, and \ref{lem_reduction_4} to reduce to the case where $G$ is either a path of length at most $2$ or a unicyclic graph in which each vertex on the cycle has at most one pure branch of length $1$ or $2$ attached to it. In the latter case, we apply Lemmas~\ref{lem_reduction_5} and \ref{lem_reduction_6} to reduce to the case of a tree, and hence to a path of length at most $2$, for which the result is clear.

We note that a given reduction step may be applied more than once. For example, after shortening a branch, several branches of the same length may arise; after removing these branches, the branching vertex itself may become part of a pure branch. However, at each step, the number of vertices decreases by at least one.
\end{proof}

\begin{proof}[Proof of Theorem~\ref{main}] When $G$ is a bipartite graph, it follows from a result of Herzog, Hibi, and Trung \cite{HHT} that $J(G)^t = J(G)^{(t)}$ for all $t$. Hence, we have $\tau_{\max}(G) = \delta(G)$; see also \cite[Theorem~4.9]{DHNT}. The conclusion then follows from Theorem~\ref{thm_reg_pow_uni}.
\end{proof}

\begin{rem}
\begin{enumerate}
    \item In general, if $H$ is a subgraph of $G$, one does not necessarily have
    \(     \tau_{\max}(H) \le \tau_{\max}(G);    \) see \cite{F3}.
    \item By \cite[Theorem 2.4]{LW},     \(    \reg\bigl(J(C_n)^t\bigr)
    =
    \reg\bigl(J(C_n)^{(t)}\bigr)
    \) where $C_n$ is an odd cycle. Furthermore,
    \(    \delta(C_n) = \tau_{\max}(C_n).    \) Hence, Theorem~\ref{thm_reg_pow_uni} yields a formula for the regularity of powers of cover ideals of cycles.
    \item Bijender, Kumar, and Kumar \cite{BKK} proved that all powers of the cover ideal of a simplicial forest are componentwise linear. In particular,
    \[
    \reg\bigl(J(G)^t\bigr) = t \tau_{\max}(G)
    \]
    for all $t \ge 1$. Consequently, the symbolic powers of cover ideals of path ideals of paths are also componentwise linear. It would be interesting to determine the regularity of symbolic powers of cover ideals of path ideals of cycles, as the corresponding problem for path ideals of cycles is itself subtle; see \cite{BCV}.
\end{enumerate}
\end{rem}

\vspace{1cm}
\noindent {\bf Data Availability} Data sharing is not applicable to this article as no datasets were generated or analyzed during the current study.

\noindent {\bf Conflict of interest} There are no competing interests of either financial or personal nature.

\end{document}